\documentclass[12pt]{article}

\usepackage{url}
\usepackage{xparse} 

\usepackage{latexsym,amssymb,upref,amsmath,amsthm, amsfonts,authblk}
\usepackage{amssymb,amsmath,amsthm, calc, graphicx}
\usepackage{epsfig}
\usepackage{breqn}
\usepackage{footnpag}
\usepackage{rotating}
\usepackage{amsfonts}
\usepackage{setspace}
\usepackage{fullpage}
\usepackage{enumitem}
\usepackage{bbold}
\usepackage{comment}
\usepackage{pgf,tikz}
\usepackage{mathrsfs}

\usetikzlibrary{decorations.markings}

\usepackage{pgf,tikz,pgfplots}
\pgfplotsset{compat=1.14}
\usepackage{mathrsfs}
\usetikzlibrary{arrows}
\usepackage{authblk}

\newtheorem{thm}{Theorem}

\newtheorem{claim}{Claim}
\newtheorem{lemma}[thm]{Lemma}

\newtheorem{observation}[thm]{Observation}

\newcommand{\thistheoremname}{}
\newtheorem*{genericthm*}{\thistheoremname}
\newenvironment{namedthm*}[1]
{\renewcommand{\thistheoremname}{#1}%
	\begin{genericthm*}}
	{\end{genericthm*}}

\newcommand{\abs}[1]{\left\lvert{#1}\right\rvert}

\title{A note on maximum size of Berge-$C_4$-free hypergraphs}

\author
{Beka Ergemlidze
\thanks{Department of Mathematics and Statistics, University of South Florida,
Tampa, Florida 33620, USA.     	E-mail: \texttt{beka.ergemlidze@gmail.com}} \qquad 

	}

\begin{document}

\maketitle

\begin{abstract}
In this paper, we consider maximum possible value for the sum of cardinalities of hyperedges of a hypergraph without a Berge $4$-cycle. We significantly improve the previous upper bound provided by Gerbner and Palmer. Furthermore, we provide a construction that slightly improves the previous lower bound.
\end{abstract}

\section{Introduction}
 
A Berge cycle of length $k$, denoted by Berge-$C_k$, is an alternating sequence of distinct vertices and distinct hyperedges of the form $v_1,h_1,v_2,h_2,\ldots v_k,h_k$ where $v_i,v_{i+1}\in h_i$ for each $i\in\{1,2,\ldots ,k-1\}$ and $v_kv_1\in h_k$.

 Throughout the paper we allow hypergraphs to include multiple copies of the same hyperedge
(multi-hyperedges).

Let $\mathcal H$ be a Berge-$C_4$-free hypergaph on $n$ vertices, Gy\H{o}ri and Lemons \cite{gyori-lemons} showed that 
$\sum_{h\in \mathcal H}(\abs{h}-3)\leq (1+o(1))12\sqrt{2} n^{3/2}$. Notice that it is natural to take $\abs{h}-3$ in the sum, otherwise we could have arbitrarily many copies of a $3$-vertex hyperedge. In \cite{gebner-palmer} Gerbner and Palmer improved the upper bound proving that  $\sum_{h\in \mathcal H}(\abs{h}-3)\leq \frac{\sqrt{6}}2 n^{3/2}+O(n)$, furthermore they showed that there exists a Berge-$C_4$-free hypergraph $\mathcal H$ such that $\sum_{h\in \mathcal H}(\abs{h}-3)\geq (1+o(1))\frac1{3\sqrt{3}} n^{3/2}.$

In this paper we improve their bounds.
\begin{thm} \label {mainthm}
Let $\mathcal H$ be a Berge $C_4$-free hypergaph on $n$ vertices, then
$$\sum_{h\in \mathcal H}(\abs{h}-3)\leq (1+o(1))\frac12 n^{3/2}.$$

Furthermore, there exists a $C_4$-free hypergraph $\mathcal H$ such that
$$(1+o(1))\frac{1}{2{\sqrt6}}n^{3/2}\leq \sum_{h\in \mathcal H}(\abs{h}-3)$$
\end{thm}

This improves the upper-bound by factor of $\sqrt{6}$ and slightly increases the lower-bound.


We introduce couple of important notations and definitions used throughout the paper.  

     Length of a path is the number of edges in the path. 

    For convenience, an edge or a pair of vertices $\{a,b\}$ is sometimes referred to as $ab$. 
    
     For a graph (or a hypergraph) $H$ , for convenience, we sometimes use $H$ to denote the edge set of the graph (hypergraph) $H$. Thus the number of edges (hyperedges) in $H$ is $\abs{H}$.
    
    


\section{Proof of Theorem \ref{mainthm}}

We will now construct a graph, existence of which is proved in \cite {gebner-palmer} (page 10).
Let us take a graph $H$ on a ground set of $\mathcal H$ by embedding edges into each hyperedge of $\mathcal H$. More specifically, for each $h\in \mathcal H$ we embed $\abs{h}-3$ edges on the vertices of $h$, such that collection of edges that were embedded in $h$ consists of pairwise vertex-disjoint triangles and edges.
We say that $e\in H$ has color $h$ if $e$ was embedded in the hyperedge $h$ of the hyergraph $\mathcal H$.
We will upper bound the number of edges in $H$, which directly gives us an upper bound on $\sum_{h\in \mathcal H}(\abs{h}-3)$.
\begin{observation}
\label{obs}
For each vertex $x$ of the graph $H$, at most $2$ adjacent edges to $x$ have the same color. Moreover, if $xy$ and $xz$ have the same color $h$, then $yz\in H$ and the color of $yz$ is $h$ as well.
\end{observation}


The following lemma is stated and proved in \cite{gebner-palmer}(claim $16$, page 10).
\begin{lemma}
\label{nok27}
$H$ is $K_{2,7}$-free.
\end{lemma}

Now we will upper bound the number of edges in $H$. It should be noted, that the only properties of $H$ that we use during the proof, are Observation \ref{obs} and Lemma \ref{nok27}.

For any vertex $v\in V(H)$, let $d(v)$ denote the degree of $v$ in the graph $H$ and
let $d$ be the average degree of the graph $H$. 

\begin{claim}
\label{maxdegree}
We may assume that a maximum degree in $H$ is less than $18\sqrt{n}$.
\end{claim}
\begin{proof}
First, using the standard argument, we will show, that we may assume minimum degree in $H$ is more than $d/3$. 
Let $u\in V(H)$ be a vertex with degree at most $d/3$. Let us delete the vertex $u$ from $H$, moreover if two distinct edges $ux,uy$ have the same color in $H$, then delete an edge $xy$ as well (by Observation \ref{obs} at most $d/6$ edges will be deleted this way). Let the obtained graph be $H'$. 
Clearly $\abs{H\setminus H'}\leq d/3+d/6$, i.e. $\abs{H'}\geq \frac{nd}2-\frac{d}2=\frac{(n-1)d}2$ and since $H'$ has $n'=n-1$ vertices, it means that the average degree of $H'$ is at least $d$, and it is easy to see that Observation \ref{obs} and Lemma \ref{nok27} still holds for $H$. So we could upper bound $H'$ in terms of $n'$ and get the same upper bound on $H$ in terms of $n$. We can repeatedly apply this procedure before we will obtain a graph, with increased (or the same) average degree, and for which Observation \ref{obs} and Lemma \ref{nok27} still holds. So we may assume, that the minimum degree in $H$ is more than $d/3$.

Let us assume there is a vertex $u$ with degree at least $18\sqrt{n}$. It is easy to see, that there are at least $18\sqrt{n}\cdot (d/3-1)$ paths of length $2$ starting at $u$, moreover each vertex of $H$ is the endpoint of at most $6$ of these $2$-paths, otherwise there would be a $K_{2,7}$, contradicting Lemma \ref{nok27}. So $n>18\sqrt{n}\cdot (d/3-1)/6$, therefore $d<\sqrt{n}+3$, i.e. $\abs{H}< n^{3/2}/2+1.5n$ and we are done. Therefore, we may assume that degree of each vertex of $H$ is less than $18\sqrt{n}$.
\end{proof}

Let $N_1(v)=\{x\mid vx\in E(H)\}$ and $N_2(v)=\{y\notin N_1(v)\cup\{v\}\mid \exists x\in N_1(v)$ s.t. $yx\in E(H)\}$ denote the first and the second neighborhood of $v$ in $H$, respectively.

Let us fix an arbitrary vertex $v$ and let $G=H[N_1(v)]$ be a subgraph of $H$ induced by the set $N_1(v)$.
Clearly, the maximum degree in $G$ is at most $6$, otherwise there is a $K_{2,7}$ in the graph $H$, which contradicts Lemma \ref{nok27}. So 
\begin{equation}
\label{eq1}
    \abs{G}\leq 3|N_1(v)|=3d(v).
\end{equation}

Let $G_{aux}$ be an auxiliary graph with the vertex set $N_1(v)$ such that $xy\in E(G_{aux})$ if and only if there exists a $w\in N_2(v)$ with $wx,wy\in E(H)$. 
Let $G_{aux}'$ be the graph with an edge st $E(G_{aux})\setminus E(G)$, clearly 

\begin{equation}
\label{eqaux}
 \abs{G_{aux}}\leq \abs {G_{aux}'}+\abs {G}\leq \abs {G_{aux}'}+3d(v).   
\end{equation}

\begin{lemma}
\label{nok55}
$\abs{G_{aux}'}<{d(v)}^{9/5}$. 
\end{lemma}
\begin{proof}
 If we show, that $G_{aux}'$ is $K_{5,5}$-free, then by Kovari-Sos-Tur\'an\cite{Kovari-sos-turan} theorem $\abs{G_{aux}'}\leq \frac{4^{1/5}}2{d(v)}^{9/5}<d(v)^{9/5}$.
 So it suffices to prove that $G_{aux}'$ is $K_{5,5}$-free

First, let us prove the following claim.
\begin{claim}
\label{inclusion}
Let $xy$ be an edge of $G_{aux}'$ and let $h_x$ and $h_y$ be the colors of $vx$ and $vy$ in $H$, respectively.  Then either $x\in h_y$ or $y\in h_x$. 
\end{claim}
\begin{proof}
First note that $h_x\not = h_y$ otherwise, by observation \ref{obs}, $xy$ would be an edge of $G$ and therefore not and edge of $G_{aux}$. By definition of $G_{aux}'$ there exists $w\in N_2(v)$ such that $wx,wy\in H$. Let $h_1$ and $h_2$ be colors of $wx$ and $wy$ respectively. $h_1\not =h_2$, otherwise $xy\in G$, a contradiction. If $h_1=h_x$ or $h_2=h_y$ then by observation \ref{obs} $wv\in E(H)$, therefore $w\in N_1(v)$, a contradiction. Clearly $h_1,h_2,h_x,h_y$ are not all distinct, otherwise they would form a Berge-$C_4$. So either $h_1=h_y$ or $h_2=h_x$, therefore $x\in h_y$ or $y\in h_x$.
\end{proof}

Now let us assume for a contradiction, that there is a $K_{5,5}$ in $G_{aux}'$ with parts $A$ and $B$. By the pigeon-hole principle, there exists $v_1,v_2,v_3\in A$ such that colors of $vv_1,vv_2$ and $vv_3$ are all different. Similarly, there exists $v_4,v_5,v_6\in B$ such that colors of $vv_4,vv_5$ and $vv_6$ are distinct. For each $1\leq i\leq 6$ let $h_i\in E(\mathcal H)$ be the color of $vv_i$. 
If $v_i\in A$, $v_j\in B$ and $h_i=h_j$, then $v_iv_j\in G$ therefore $v_iv_j\notin G_{aux}$, a contradiction. So $h_i$ is different for each $i\in \{1,2,3,4,5,6\}$.
So we have a $K_{3,3}$ in $G_{aux}'$ with parts $v_1,v_2,v_3$ and $v_4,v_5,v_6$ such that color $h_i$ of each $vv_i$ is distinct for each $1\leq i \leq 6$.

Let $D$ be a bipartite directed graph with parts $v_1,v_2,v_3$ and $v_4,v_5,v_6$, such that $\vec{v_iv_j}\in D$ if and only if $v_i\in h_j$ and $v_i$ and $v_j$ are in different parts. By Claim $\ref{inclusion}$ for each $1\leq i \leq 3$ and $4\leq j\leq 6$, either $\vec{v_iv_j}\in D$ or $\vec{v_jv_i}\in D$.


\begin{claim}
\label{directedgraph}
Let $F_1$ and $F_2$ be directed graphs with the edge sets $E(F_1)=\{\vec{yx},\vec{zx},\vec{wz}\}$ and $E(F_2)=\{\vec{yx},\vec{zx},\vec{zw},\vec{uw}$\}, where $x,y,z,w,u$ are distinct vertices. Then $D$ is $F_1$-free and $F_2$-free. 
\end{claim}


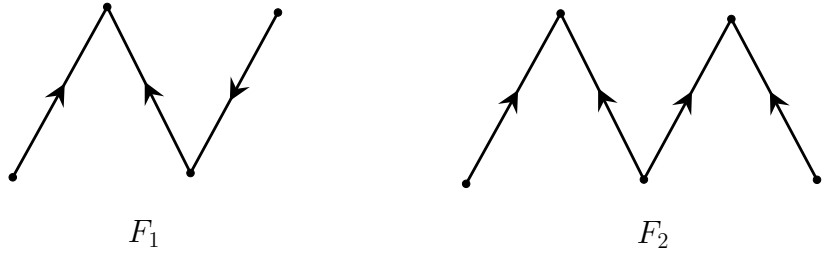
\begin{centering}
\resizebox{.666\textwidth}{!}{

\begin{tikzpicture}[font=\Large, line cap=round,line join=round,>=stealth,x=4.5cm,y=4.5cm]

\draw [line width=1.5pt,decoration={markings,mark=at position 0.55 with {\arrow[scale=2]{>}}},
    postaction={decorate}] (-7.480130220401499,5.094025952733748) -- (-7.0858666666667,5.805528644997276);
\draw [line width=1.5pt,decoration={markings,mark=at position 0.55 with {\arrow[scale=2]{>}}},
    postaction={decorate}] (-6.73824677055632,5.111875890725895) -- (-7.0858666666667,5.805528644997276);
\draw [line width=1.5pt,decoration={markings,mark=at position 0.55 with {\arrow[scale=2]{>}}},
    postaction={decorate}] (-6.373195735016119,5.782408798801786) -- (-6.73824677055632,5.111875890725895);
\draw [line width=1.5pt,decoration={markings,mark=at position 0.55 with {\arrow[scale=2]{>}}},
    postaction={decorate}] (-5.587772837866885,5.0668130633154735) -- (-5.1935092841320865,5.778315755579001);
\draw [line width=1.5pt,decoration={markings,mark=at position 0.55 with {\arrow[scale=2]{>}}},
    postaction={decorate}] (-4.845889388021707,5.084663001307621) -- (-5.1935092841320865,5.778315755579001);
\draw [line width=1.5pt,decoration={markings,mark=at position 0.55 with {\arrow[scale=2]{>}}},
    postaction={decorate}] (-4.123372706791822,5.083597018972495) -- (-4.480838352481506,5.755195909383511);
\draw [line width=1.5pt,decoration={markings,mark=at position 0.55 with {\arrow[scale=2]{>}}},
    postaction={decorate}] (-4.845889388021707,5.084663001307621) -- (-4.480838352481506,5.755195909383511);
\draw (-7.031458022842719,4.947453627782717) node[anchor=north west] {\textit{$F_1$}};
\draw (-4.90719548951592,4.944315869240129) node[anchor=north west] {\textit{$F_2$}};
\begin{scriptsize}
\draw [fill=black] (-7.480130220401499,5.094025952733748) circle (2pt);
\draw [fill=black] (-7.0858666666667,5.805528644997276) circle (2pt);
\draw [fill=black] (-6.73824677055632,5.111875890725895) circle (2pt);
\draw [fill=black] (-6.373195735016119,5.782408798801786) circle (2pt);
\draw [fill=black] (-5.587772837866885,5.0668130633154735) circle (2pt);
\draw [fill=black] (-5.1935092841320865,5.778315755579001) circle (2pt);
\draw [fill=black] (-4.845889388021707,5.084663001307621) circle (2pt);
\draw [fill=black] (-4.480838352481506,5.755195909383511) circle (2pt);
\draw [fill=black] (-4.123372706791822,5.083597018972495) circle (2pt);
\end{scriptsize}
\end{tikzpicture}

}

\end{centering}


\begin{proof}
Let us assume, that $D$ contains $F_1$. Then without loss of generality we may assume, that $\vec{v_4v_1},\vec{v_5v_1},\vec{v_2v_5}\in D$.
So by definition of $D$, $v_4,v_5\in h_1$ and $v_2\in h_5$. Then we have, $vv_4\subset h_4$, $v_4v_5\subset h_1$, $v_5v_2\subset h_5$ and $v_2v\subset h_2$, therefore the hyperedges $h_4,h_1,h_5,h_2$ form a berge $C_4$ in $\mathcal H$, a contradiction.

If $D$ contains $F_2$, without loss of generality we may assume that $\vec{v_4v_1},\vec{v_5v_1},\vec{v_5v_2},\vec{v_6v_2}\in D$.
So by definition of $D$, we have $v_4,v_5\in h_1$ and $v_5,v_6\in h_2$, so $v,h_4,v_4,h_1,v_5,h_2,v_6,h_6$ is a Berge-$C_4$, a contradiction.
\end{proof}

Now since each vertex of $D$ has at least $3$ incident edges, there is a vertex in $D$ with at least $2$ incoming edges, without loss of generality let this vertex be $v_1$ and let the incoming edges be $\vec{v_4v_1}$ and $\vec{v_5v_1}$. By Claim \ref{directedgraph} $D$ does not contain $F_1$, therefore for each $2\leq i \leq 3$, $\vec{v_iv_4},\vec{v_iv_5}\notin D$ i.e. $\vec{v_4v_i},\vec{v_5v_i}\in D$ for every $1\leq i \leq 3$. If $\vec{v_6v_1}\in D$, then $\vec{v_4v_2},\vec{v_5v_2},\vec{v_5v_1},\vec{v_6v_1}$ would form $F_2$ which contradicts Claim \ref{directedgraph}, therefore $\vec{v_1v_6}\in D$. Similarly $\vec{v_2v_6}\in D$ (and $\vec{v_3v_6}\in D$), but now $\vec{v_1v_6},\vec{v_2v_6}$ and $\vec{v_4v_1}$ form $F_1$, a contradiction.


This completes the proof of the lemma.
\end{proof}
Using the information above, we will complete the proof of the upper bound.
By the Blackley-Roy inequality there exists a vertex $v\in V(H)$ such that there are at least $d^2$ ordered $2$-walks starting at vertex $v$. We now fix this vertex $v$ and define $G$, $G_{aux}$ and $G'_{aux}$ for $v$ similarly as before. Clearly at most $2d(v)$ of these $2$-walks may not be a path, so there are at least $d^2-2d(v)$ $2$-paths starting at $v$. 

Let $B$ be a bipartite graph with parts $N_1(v)$ and $N_2(v)$ such that $xy\in B$ if and only if $vxy$ is a $2$-path of $H$ and $y\in N_2(v)$ (clearly $x\in N_1(v)$). The number of $2$ paths $vxy$ such that $xy\notin B$ is exactly $2\abs{G}\leq 6d(v)$ (here we used \eqref{eq1}), therefore we have $$|B|\geq d^2-2d(v)-6d(v)=d^2-8d(v).$$

 Let $B'$ be a subgraph of $B$ with the edge set $E(B')=\{xy\in E(B)\mid \exists z\in N_1(v)\setminus \{x\}$ such that $yz\in E(B)\}$. Clearly, $xy,yz\in E(B')$ means that $xz\in E(G_{aux})$, moreover, by Lemma \ref{nok27}, for each $xz\in G_{aux}$ there is at most $6$ choices of $y\in N_2(v)$ such that $xy,yz\in E(B')$, therefore it is easy to see that the number of $2$-paths in $B'$ with terminal vertices in $N_1(v)$ is at most $6\cdot \abs{G_{aux}}$, So $\abs{B'}\leq 12\abs{G_{aux}}$, therefore by Lemma \ref{nok55} and \eqref{eqaux} we have $\abs{B'}\leq 12(d(v)^{9/5}+3d(v))$, so 
\begin{equation}
    \label{boundofb}
    \abs{B\setminus B'}\geq d^2-12d(v)^{9/5}-44d(v)
\end{equation}

On the other hand, by definition of $B'$ each vertex of $N_2(v)$ is incident to at most $1$ edge of $B\setminus B'$, so $\abs{N_2(v)}\geq \abs{B\setminus B'}$, therefore by \eqref{boundofb} we have $n>\abs{N_2(v)}\geq d^2-12d(v)^{9/5}-44d(v)$.
Using Claim \ref{maxdegree} we have $d^2<n+12\cdot (18\sqrt{n})^{9/5}+44\cdot 18\sqrt{n}$ i.e. $d^2<n+2184n^{0.9}+792\sqrt{n}$. So for large enough $n$ we have $$d<\sqrt{n}+1100n^{0.4}.$$
Therefore $$\abs{H}< \frac{1}{2}n^{2.5}+550n^{1.4}=\frac{1}{2}n^{1.5}(1+o(1))$$
$$\sum_{h\in \mathcal H}(\abs{h}-3)=\abs{H}\leq (1+o(1))\frac12 n^{3/2}.$$

Now it remains to prove the lower bound. Let $G$ be a bipartite $C_4$-free graph on $n/3$-vertices, with $\abs{E(G)}=\left ( \frac{n}{6}\right )^{3/2}+o(n^{3/2})$ edges. Let us replace each vertex of $G$ by $3$ identical copies of itself, this will transform each edge to a $6$-set. Let the resulting $6$-uniform hypergraph be $\mathcal H$. Clearly $$\sum_{h\in \mathcal H}(\abs{h}-3)=3\abs{\mathcal H}=\frac{1}{2{\sqrt6}}n^{3/2}+o(n^{3/2}).$$
Now let us show that $\mathcal H$ is Berge-$C_4$-free. Let us assume for a contradiction that there is a Berge $4$-cycle in $\mathcal H$, and let this Berge cycle be $a,h_{ab},b,h_{bc},c,h_{cd},d,h_{da}$.
If  $a,b,c,d$ are copies of $4$ or $3$ distinct vertices of $G$, then there would be a $C_4$ or $C_3$ in $G$ respectively, a contradiction. So $a,b,c,d$ are copies of only two vertices of $G$, say $x$ and $y$, so at least two of the pairs $ab$,$bc$,$cd$,$da$ correspond to $xy$ in $G$, therefore it is easy to see, that at least two of the hyperedges $h_{ab},h_{bc},h_{cd},h_{da}$ should be the same, a contradiciton.

\section*{Acknowledgements}
I want to thank my colleagues Abhishek Methuku and Ervin Gyori for having extremely helpful discussions about this problem.

\end{document}